\documentclass{proc-l}

\newtheorem{theorem}{Theorem}[section]

\newtheorem{cor}[theorem]{Corollary}
\newtheorem{prop}[theorem]{Proposition}
\newtheorem{alg}[theorem]{Algorithm}
\newtheorem{notation}[theorem]{Notation}

\theoremstyle{definition}
\newtheorem{definition}[theorem]{Definition}

\theoremstyle{remark}

\numberwithin{equation}{section}

\begin{document}

\title[Recursive Relations for 2-Variable Weighted Shifts]{Recursive Relations \\ for 2-Variable Weighted Shifts}


%    author one information
\author{Edward L. White}
\address{}
\curraddr{}
\email{}
\thanks{}

\date{}

\dedicatory{}

%    "Communicated by" -- provide editor's name; required.

%    Abstract is required.
\begin{abstract}
In this paper, we outline a method to determine all recursive relations for a subnormal 2-variable weighted shift, up to total degree $k$, entirely from the representing measure. This allows us to show that the densities of the atoms do not affect the recursive relations. We then show that this method can be used to explicitly construct the Gröbner basis for the ideal of recursive relations.
\end{abstract}

\maketitle

\section{Introduction}
Let $\mathbb{R},\mathbb{R}_{+},\mathbb{Z},\text{ and }\mathbb{Z}_{+}(\equiv \mathbb{N}_0)$ denote the real numbers, positive real numbers, integers, and nonnegative integers, respectively. Given a bounded sequence of positive real numbers $\alpha \equiv \{\alpha\}_{n=0}^{\infty}$, let $W_{\alpha}:\ell^2(\mathbb{Z}_{+}) \rightarrow \ell^2(\mathbb{Z}_{+})$ denote the \textit{unilateral weighted shift} defined by $W_{\alpha} e_n = \alpha_n e_{n+1}$, for all $ n \in \mathbb{Z}_{+}$ where $\{e_n\}_{n=0}^\infty$ is the canonical orthonormal basis for $\ell^2(\mathbb{Z}_{+})$. The moments of $\alpha$ are given as

$$
\gamma_k \equiv \gamma_k(\alpha):=\left\{\begin{array}{cc}
1 & \text { if } k=0 \\
\alpha_0^2 \cdot \ldots \cdot \alpha_{k-1}^2 & \text { if } k>0
\end{array}\right\} .
$$

We recall that an operator is \textit{normal} if it commutes with its adjoint and  \textit{subnormal} if it is the restriction of a normal operator to an invariant subspace. We also recall the well known characterization of subnormality for a unilateral weighted shift \cite[III.8.16]{conway1991theory} \cite{gellar1970subnormal}. Explicitly, $W_\alpha$ is subnormal if and only if there exists a probability measure $\mu$ supported in $\left[0,\left\|W_\alpha\right\|^2\right]$ such that
$$
\gamma_k(\alpha):=\alpha_0^2 \cdots \alpha_{k-1}^2=\int t^k d \mu(t)(k \geq 1) 
$$
where $\left\|W_\alpha\right\|=\sup_k \alpha_k.$

Let $\boldsymbol{\alpha_k}, \boldsymbol{\beta_k} \in \ell^{\infty}(\mathbb{Z}^2_+)$ be two double-indexed positive bounded sequences,  for $\boldsymbol{k}:=(k_1,k_2) \in \mathbb{Z}^2_+$. Let $\ell^2\left(\mathbb{Z}_{+}^2\right)$ be the Hilbert space of square-summable complex sequences indexed by $\mathbb{Z}_+^2$ with orthonormal basis $\{e_{\boldsymbol{k}}\}_{\boldsymbol{k}\in \mathbb{Z}^2_+}$. (Recall $\ell^2(\mathbb{Z}_+^2)$ is isometrically isomorphic to $\ell^2(\mathbb{Z_+})\otimes\ell^2(\mathbb{Z_+}).$) We define the \textit{2-variable weighted shift} $\mathbf{T} \equiv\left(T_1, T_2\right)$ by
$$
\begin{aligned}
& T_1 e_{\boldsymbol{k}}:=\alpha_{\boldsymbol{k}} e_{\boldsymbol{k}+\varepsilon_1} \\
& T_2 e_{\boldsymbol{k}}:=\beta_{\boldsymbol{k}} e_{\boldsymbol{k}+\varepsilon_2},
\end{aligned}
$$
where $\varepsilon_1:=(1,0)$ and $\varepsilon_2:=(0,1)$.  Given $\boldsymbol{k} \in \mathbb{Z}_{+}^2$, the moment of $(\boldsymbol{\alpha}, \boldsymbol{\beta})$ of order $\boldsymbol{k}$ is
$$
\gamma_{\boldsymbol{k}}(\mathbf{T}) :=\left\{\begin{array}{ll}
1 & \text { if } \boldsymbol{k}=(0,0) \\
\alpha_{(0,0)}^2 \cdots \alpha_{\left(k_1-1,0\right)}^2 & \text { if } k_1 \geq 1 \text { and } k_2=0 \\
\beta_{(0,0)}^2 \cdots \beta_{\left(0, k_2-1\right)}^2 & \text { if } k_1=0 \text { and } k_2 \geq 1 \\
\alpha_{(0,0)}^2 \cdots \alpha_{\left(k_1-1,0\right)}^2 \beta_{\left(k_1, 0\right)}^2 \cdots \beta_{\left(k_1, k_2-1\right)}^2 & \text { if } k_1 \geq 1 \text { and } k_2 \geq 1
\end{array}\right.
$$

There is a similar characterization for subnormality of a 2-variable weighted shift given in \cite{jewell1979commuting}. Explicitly, a 2-variable weighted shift $\mathbf{T} \equiv\left(T_1, T_2\right)$ is subnormal if and only if there is a probability measure $\mu$ defined on the 2-dimensional rectangle $R:=\left[0, \left\|T_1\right\|^2\right] \times\left[0,\left\|T_2\right\|^2\right]$ such that: 

$$\gamma_{\boldsymbol{k}}(\mu):=\int_R \mathbf{t}^{\boldsymbol{k}} d \mu(\mathbf{t}):=\int_R t_1^{k_1} t_2^{k_2} d \mu\left(t_1, t_2\right) \quad \text{ for all } \boldsymbol{k} \in \mathbb{Z}_{+}^2.$$
Further, we say $\mu$ is a representing measure for $\mathbf{T}.$ We will use the notation $(t_1,t_2)$ and $(x,y)$ interchangeably.

Given a subnormal 2-variable weighted shift with representing measure $\mu$, we recall the moment matrix associated to $\mu,$ introduced in \cite{curto1996solution,curto1998flat}
$$M(\mu)(k):=
\begin{bmatrix}
    \int1 d\mu & \int x d \mu & \int y d \mu & \cdots & \int y^{k} d\mu\\
    \int x d \mu & \int x^{2} d \mu & \int xy d \mu & \cdots & \int xy^{k}  d\mu\\
    \int y d \mu & \int xy d \mu & \int y^2 d \mu & \cdots & \int y^{k+1}  d\mu\\
    \vdots & \vdots & \vdots & \ddots  & \vdots\\
     \int y^n d \mu & \int xy^n d \mu & \int y^{n+1} d \mu & \cdots &\int y^{2k}  d\mu\\
\end{bmatrix}.
$$
Note the above definition does differ from the traditional definition of the moment matrix for $\mathbf{T}$, which has the $(i,j)$ entry as $\gamma_{(i,j)}$ or $M(\gamma)(k):=\left(\gamma_{(i,j)}\right)_{i,j \geq 0, i+j \leq 2k}.$ However, in this paper we will only be concerned with subnormal 2-variable weighted shifts so, in this case, the definitions are equivalent. We also define $M(\mu) \equiv M(\mu)(\infty)$ to denote the infinite dimensional moment matrix.

Let $\delta(x,y)$ denote the point-mass probability measure with support $(x,y)\in\mathbb{R}^2$. It is clear that for any $(x,y)\in \mathbb{R}^2$, $\operatorname{rank}[M(\delta(x,y))(k)]=1.$ Further, $M(\delta(x,y))(k)$ is positive semidefinite as it is a Gram matrix. We say $\mu$ is \textit{finitely atomic probability measure} if $\mu
\equiv\sum\limits_{i=1}^n \lambda_i \delta(x_i,y_i)$ where $\lambda_i\in \mathbb{R_+.}$ Additionally, if $\mu$ is a finitely atomic probability measure, then $\operatorname{card} \operatorname{supp}(\mu) \geq \operatorname{rank}(M(\mu)(n))$ (\cite[7.6]{curto1996solution}).

Now if we label our columns of $M(\mu)(k)$ via lexicographic ordering (i.e.

\noindent $1,X,Y,X^2,XY,Y^2, X^3\cdots)$ then we can associate column relations with bivariate polynomials with total degree at most $k$, which will be denoted as $\mathbb{R}_k[x,y].$ For $p \in \mathbb{R}_k[x,y]$, $p(x,y) \equiv \sum\limits_{i,j\geq 0, i+j \leq k} a_{ij}x^iy^j.$ Let $\hat{p}:=(a_{ij})$ denote the vector of coefficients with respect to the basis of $\mathbb{R}_k[x,y]$ consisting of monomials in lexicographic order. We denote a linear combination of columns by $p(X,Y) := [M(\mu)(k)] \cdot\hat{p}$ for $p\in \mathbb{R}_k[x,y].$ We say  $M(\mu)(k)$ is \textit{recursively generated} if the kernel of $M(\mu)(k)$, denoted  $\operatorname{ker}(M(\mu)(k))$, has the following ideal-like property:
    $$p,q,pq \in \mathbb{R}_k[x,y], \quad p(X,Y)=\boldsymbol{0} \implies (pq)(X,Y)=\boldsymbol{0}$$
where $\boldsymbol{0}$ denotes the zero vector. Further, we say $p$ is a \textit{recursive relation} for $M(\mu)(k).$  It should also be noted that if $p \in \mathbb{R}_k[x,y]$ is a recursive relation for $M(\mu)(k),$ then it will also be a recursive relation for $M(\mu)$ due to the Flat Extension Theorem \cite[Thm. 5.13]{curto1996solution}.  When looking at recursive relations as bivariate polynomials instead of column relations, the \textit{algebraic variety} of $M(\mu)(k)$ is the intersection of the zero sets of all recursive relations of $M(\mu)(k).$ As shown in \cite{curto2005truncated}, the existence of a representing measure for a moment sequence implies that $M(\mu)(k)$ is positive semidefinite and recursively generated for all integers $k\geq 1$.

It is shown in \cite[Cor 2.2]{laurent2005revisiting} that the recursive relations for a moment matrix generate a polynomial ideal. Thus, a useful tool for analyzing the recursive relations will be Gröbner bases. If $\mathcal{I}$ is the polynomial ideal generated by the polynomials $f_1,f_2,\cdots, f_n$, denoted $\mathcal{I}:= \langle  f_1,f_2,\cdots, f_n\rangle$, we say $\mathcal{G}\subseteq\mathcal{I}$ is a Gröbner basis for $\mathcal{I}$ if the leading terms, with respect to a given ordering, are sufficient to generate all the leading terms of polynomials in $\mathcal{I}.$ In this paper, we will exclusively be using the lexicographic ordering. There are several advantages of examining a Gröbner basis of recursive relations. First, a Gröbner basis for $\mathcal{I}$ is a basis for $\mathcal{I}$ \cite[Ch 2 \S 5 Cor. 6]{cox1997ideals}. Second, they allow us to determine which polynomials are elements of $\mathcal{I}.$ Thus, given a moment matrix, finding the Gröbner basis would not only allow us to generate all the recursive relations, but also allow us to determine if a proposed recursive relation holds for the moment matrix. In this paper, we present a method for calculating Gröbner bases for recursive relations for 2-variable weighted shifts from their finitely atomic representing measures.

\section{Notation and Preliminary}
In this section we briefly recall some notation and preliminary results that will be needed in the sequel.
\begin{prop}[\cite{cox1997ideals}, Ch 1 \S 4 Prop. 8]
    
Let $V,W \subseteq \mathbb{R}^2$ be algebraic varieties and define $\textbf{I}(V):=\{f\in\mathbb{R}[x,y]:f(x,y)=0\}.$ Then $V=W$ if and only if $\mathbf{I}(V)=\mathbf{I}(W).$
\end{prop}

\begin{notation}
    Following the notation in \cite[Ch 2 \S 6 Def. 3]{cox1997ideals}, let $\overline{f}^F$ denote the remainder on division of $f\in \mathbb{R}[x,y]$ by the ordered $s$-tuple $F:=(f_1,f_2,\cdots,f_s)\subset\mathbb{R}[x,y]$. For $f\in \mathbb{R}[x,y]$ let $\operatorname{LT}(f)$ denote the leading term of $f$ and $\operatorname{LM}(f)$ denote the leading monomial with respect to the chosen ordering.   
\end{notation}

\begin{definition}
        Again following the notation from \cite[Ch 2 \S 6 Def. 4]{cox1997ideals}, given non-zero $f, g \in \mathbb{R}[x,y]$, the \textit{$S$-polynomial} of $f$ and $g$ is

            $$
                S(f, g)=\frac{L}{\operatorname{LT}(f)} f-\frac{L}{\operatorname{LT}(g)} g
                $$

            where $L=\operatorname{lcm}(\operatorname{LM}(f), \operatorname{LM}(g))$.
    \end{definition}
\begin{samepage}
\begin{alg}[Buchberger's Algorithm \cite{cox1997ideals}] Let $\mathcal{I}=\left\langle f_1, \ldots, f_s\right\rangle \neq\{0\}$ be a polynomial ideal. Then a Gröbner basis for $\mathcal{I}$ can be constructed in a finite number of steps by the following algorithm:

Input: $F=\left(f_1, \ldots, f_s\right)$

Output : a Gröbner basis $G=\left(g_1, \ldots, g_t\right)$ for $I$, with $F \subseteq G$
\\

$
G:=F
$

REPEAT

$ \quad \quad
G^{\prime}:=G
$

\quad \quad FOR each pair $\{p, q\}, p \neq q$ in $G^{\prime}$ DO

$$
\begin{array}{l}
r:=\overline{S(p, q)} ^{G^{\prime}} \\
\text{IF } r \neq 0 \text { THEN } G:=G \cup\{r\}
\end{array}
$$

UNTIL $G=G^{\prime}$

RETURN $G.$
    
\end{alg}
\end{samepage}

\begin{theorem}[Division Algorithm  \cite{cox1997ideals}]Fix a monomial order $>$ on $\mathbb{Z}_{+}^n$, and let $F=\left(f_1, \cdots, f_s\right)$ be an ordered s-tuple of polynomials in $\mathbb{R}\left[x_1, \cdots, x_n\right]$. Then every $f \in \mathbb{R}\left[x_1, \cdots, x_n\right]$ can be written as

$$
f=a_1 f_1+\cdots+a_s f_s+r
$$

\noindent where $a_i, r \in \mathbb{R}\left[x_1, \cdots, x_n\right]$, and either $r=0$ or $r$ is a linear combination of monomials with coefficients in $\mathbb{R}$, none of which is divisible by any leading terms of $f_1, \cdots, f_s$. We call $r$ a remainder of $f$ on division by $F$. 
\end{theorem}
\section{Finding Recursive Relations}
We first note that the moment matrix corresponding to $\mu = \sum\limits_{i=1}^n\lambda_i \delta(x_i,y_i)$, $M(\mu)(k),$ where $\lambda_i\in\mathbb{R}_{+}$ can expressed as 

$$M(\mu)(k) = \sum\limits_{i=1}^n\lambda_i M(\delta(x_i,y_i))(k).$$

This relationship will allow us to view the recursive relations of $M(\mu)(k)$ by examining the recursive relations for $M(\delta(x_i,y_i))(k),$ as shown in the next proposition.

\begin{notation}
    Let $p_{\mu}(X,Y):=[M(\mu)] \cdot\hat{p}$  and  $p_{(a,b)}(X,Y):=[M(\delta_{(a,b)})] \cdot\hat{p}$. 
\end{notation}
    
    %\cdot\hat{p}$ denote the linear combination of columns $p(X,Y)$ with $X$ and $Y$ the columns of $M(\mu)$. Similarly, let $p_{(a,b)}(X,Y)$ denote $p(X,Y)$ with $X$ and $Y$ the columns of $M(\delta_{(a,b)})$.

\begin{prop}
    Let $\mu$ be an $n$-atomic probability measure with atoms $(x_i,y_i)$ for $1 \leq i \leq n$. Then $p_{\mu}(X,Y)=\boldsymbol{0}$ if and only if $p_{(x_i,y_i)}(X,Y)=\boldsymbol{0}$ for all $1 \leq i \leq n.$
\end{prop}

\begin{proof}
     We begin with $M(\mu) = \sum\limits_{i=1}^n\lambda_i M(\delta(x_i,y_i))$; this clearly implies the reverse implication.
    
    To prove the forward implication, assume $p_{\mu}(X,Y)=\boldsymbol{0}$ (i.e. $[M(\mu)] \cdot\hat{p} = \boldsymbol{0}$). Without loss of generality assume that the total degree of $p(x,y)$ is $k.$ It is then sufficient to prove that $[M(\mu)(k)] \cdot\hat{p} = \boldsymbol{0}$ implies $[M(\delta(x_i,y_i)(k)] \cdot\hat{p} = \boldsymbol{0}$ for all $1\leq i\leq n.$  Note each matrix in the sum decomposition of $M(\mu)(k)$ will be positive semidefinite as each 1-atomic moment matrix is a Gram matrix. Moreover, linear combinations of positive semidefinite matrices with positive scalars are positive semidefinite. Finally, it is well know that if $A$ and $B$ are two $m \times m$ positive semidefinite matrices, then $\operatorname{ker}(A + B)=\operatorname{ker}(A) \cap \operatorname{ker}(B).$ We then apply this results inductively $n-1$ times and see $p_{(x_i,y_i)}(X,Y)=\boldsymbol{0}$ for all $1 \leq i \leq n,$ giving the desired condition.
\end{proof}
The importance of the above proposition resides in that it tells us the atoms do not cancel each other out. Thus, to find recursive relations for $M(\mu)(k)$ it is sufficient to find the intersection of sets of recursive relations for all $M(\delta{(x_i,y_i)})(k).$  

\begin{definition}
    Let $\mu$ be the finitely atomic probability measure $\mu:= \sum\limits_{i=1}^n\lambda_i \delta(x_i,y_i).$ Then we say

    $$V(\mu)(k) := \begin{bmatrix}
        1 & x_1 & y_1 & x_1^2 & x_1 y_1 & \cdots& x_1 y_1^{k-1}& y_1^{k} \\
        1 & x_2 & y_2 & x_2^2 & x_2 y_2 & \cdots& x_2 y_2^{k-1} &y_2^{k} \\ 
       1 & x_3 & y_3 & x_3^2 & x_3 y_3 & \cdots& x_3 y_3^{k-1}& y_3^{k} \\
        \vdots & \vdots & \vdots & \vdots & \vdots &\ddots & \vdots& \vdots\\
        1 & x_n & y_n & x_n^2 & x_n y_n & \cdots & x_n y_n^{k-1}&y_n^{k}\\
    \end{bmatrix}$$
    is the \textit{2-variable Vandermonde-like matrix} for $\mu.$
    
\end{definition}
It should be noted that $V(\mu)(k)$ need not be square. Hence, this matrix is not invertible and we cannot use the techniques of polynomial interpolation as done in \cite[Thm. 2.13]{kimsey2013truncated}.  Rather, $V(\mu)(k)$ should be viewed as the first rows of all moment matrices in the sum decomposition of $M(\mu)(k)$ (i.e. the first rows of all $M(\delta(x_i,y_i)(k))$.) Since all $M(\delta{(x_i,y_i)}(k))$ must have rank 1, determining recursive relations for $M(\mu)(k)$ can be reduced to finding the kernel of $V(\mu)(k).$
 
\begin{theorem}
    Let $\mu$ be a finitely atomic probability measure. Then $p\in \mathbb{R}_k[x,y]$ is a recursive relation for $M(\mu)(k)$ if and only if $\hat{p} \in \operatorname{ker}(V(\mu)(k)).$ 
\end{theorem}

\begin{proof}
    We proceed by double inclusion to show $\operatorname{ker}(M(\mu)(k)) = \operatorname{ker}(V(\mu)(k))$. Let $\hat{p} \in \operatorname{ker}(V(\mu)(k))$, then 
    
    $$[M(\mu)(k)]\cdot \hat{p} = \sum\limits_{i=1}^n\lambda_i [M(\delta(x_i,y_i))(k)]\cdot \hat{p}.$$

    Since all $M(\delta(x_i,y_i))(k)$ are rank 1, if the  first entry of $[M(\delta(x_i,y_i))(k)] \cdot \hat{p}$ is zero, then $[M(\delta(x_i,y_i))(k) ]\cdot \hat{p}=\boldsymbol{0}$. Since this is true for all $1\leq i\leq n$ we have $[M(\mu)(k)] \cdot \hat{p} =\boldsymbol{0}$ and $\hat{p} \in \operatorname{ker}(M(\mu)(k)).$
    
    Next let $\hat{q} \in \operatorname{ker}(M(\mu)(k))$; then by Proposition 3.2 we know $\hat{q} \in \operatorname{ker}(M(\delta_{x_i,y_i})(k))$ for all $1\leq i \leq n.$ Thus the first entry of $[M(\delta_{x_i,y_i})(k)] \cdot \hat{q}$ must be zero, so $\hat{q} \in \operatorname{ker}(V(\mu)(k))$, which yields $\operatorname{ker}(M(\mu)(k))=\operatorname{ker}(V(\mu)(k))$ as desired.

    %Further since the kernel of a linear transformation is a subspace, we need to only find a basis for $V(k).$

\end{proof}
Thus finding the null space of $V(\mu)(k)$ will allow us to determine the recursive relations of $M(\mu)(k).$ Also note that in the above theorem and proof, the only requirement for the densities of the atoms is that they are positive. Moreover, the densities are not included in the 2-variable Vandermonde-like matrix. Therefore, we have the following corollary.

\begin{cor}
    The densities of a finitely atomic representing measure do not change the recursive relations of the corresponding moment matrix.
\end{cor}
     
Next, we will use Theorem 3.4 to create a Gröbner basis for the recursive relations of a moment matrix with finitely atomic representing measure.

\begin{prop}
        Theorem 3.4 can be used to construct a Gröbner basis for the recursive relations for a given finitely atomic measure.
\end{prop}
\begin{proof}
    Without loss of generality, assume that $\mu$ is an $n$-atomic measure. Then let $k$ be such that $n \leq \frac{(k+1)(k+2)}{2}.$ We then employ Theorem 3.4 for recursive relations with total degree at most $k+1$. By our choice of $k$, all monomials of total degree $k+1$ will be the leading monomials for some recursive relation as $\operatorname{card}\operatorname{supp} \mu \geq\operatorname{rank}( M(\mu)(k+1)).$ We know from \cite[Cor 2.2]{laurent2005revisiting} that the kernel of $M(\mu)$ is an ideal, denote this by $\mathcal{I} \subseteq \mathbb{R}[x,y].$ Let $\mathcal{I}' \subseteq \mathbb{R}[x,y]$ be the ideal generated by the polynomials that correspond to the basis of the null space for $V(\mu)(k+1)$. From Theorem 3.4 we see $\mathcal{I}'\subseteq \mathcal{I}$ as by construction all elements in $\mathcal{I}'$ will correspond to elements of $\operatorname{ker}([M(\mu)]).$ Next let $f\in \mathcal{I},$ if $\operatorname{deg}(f)\leq k+1,$ then $f \in \mathcal{I}'$ by Theorem 3.4. Alternatively, if $\operatorname{deg}(f) > k+1,$ we can use the Division Algorithm (Theorem 2.5) to reduce the total degree to less than or equal to $k+1,$ thus $\mathcal{I}'=\mathcal{I}$. We can then apply Buchberger's Algorithm (Algorith 2.4) to create the Gröbner basis for $\mathcal{I}.$

\end{proof}

 Propositions 2.1 and 3.6 allow us to deduce that each set of atoms corresponds to a unique Gröbner basis of recursive relations. Alternatively, if we were given a Gröbner basis for an ideal in $\mathbb{R}[x,y]$, then we could determine the algebraic variety simply by determining the zeros of all the polynomials in the Gröbner simultaneously. We then check if all elements in the algebraic variety are in elements of $\mathbb{R}_+^2$. If so, we can create a 2-variable weighted shift with recursive relations that are exactly those generated by the Gröbner basis. Without loss of generality, assume there are $n$ points in the algebraic variety. Let $\{(x_i,y_i)\}_{1 \leq i\leq n}$ denote these points and choose any set of $n$ positive real numbers $\{\lambda_i\}_{i=1}^n$ such that $\sum\limits_{i=1}^n\lambda_i=1.$ Then a measure that has the given recursive relations will be $\sum\limits_{i=1}^n\lambda_i \delta(x_i,y_i).$ Unfortunately, due to Corollary 3.5, we cannot determine the measure explicitly, we can only determine the atoms from the recursive relations. This is also seen by the fact that any set of $n$ positive real numbers $\{\lambda_i\}_{i=1}^n$ such that $\sum\limits_{i=1}^n\lambda_i=1$ will be sufficient.

\section{An Example}
In this section we give a basic example to illustrate the above results. 

Let $\mu:= \frac{1}{2}\delta(2,4)+\frac{1}{3}\delta(3,2) + \frac{1}{6}\delta(1,1)$. Then 

$$V(\mu)(2)=  \begin{bmatrix}
1 & 2 & 4 & 4 & 8 & 16 \\
1 & 3 & 2 & 9 & 6 & 4 \\
1 & 1 & 1 & 1 & 1 & 1
\end{bmatrix}$$

and 

$$\operatorname{ker}(V(\mu)(2))=\begin{bmatrix}
16 & 6 & -27 & 0 & 0 & 5 \\
12 & -8 & -9 & 0 & 5 & 0 \\
14 & -21 & 2 & 5 & 0 & 0
\end{bmatrix}.$$

Theorem 3.4 and $\operatorname{ker}(V(\mu)(2))$ imply that the recursive relations for $M(\mu)(2)$ are
\begin{align*}
    &14 - 21 x +2y+ 5x^2,\\
    &12 - 8 x -9y+ 5xy, \text{ and}\\
    &16 +6 x -27y+ 5y^2. 
\end{align*}

We can use these recursive relations and the ideal-like property to find additional recursive relations for  $M(\mu)(3).$ In fact those recursive relations can be used to generate $M(\mu)(3)$ from $M(\mu)(2).$ For example 

$$p(X,Y):=X(14 - 21 X +2Y+ 5X^2) = 14X - 21 X^2 +2XY+ 5X^3$$

will be a recursive relation. We can see this as $M(\mu)(3)\cdot\hat{p} = \mathbf{0},$

$$\begin{bmatrix}

1 & \frac{13}{6} & \frac{17}{6} & \frac{31}{6} & \frac{37}{6} & \frac{19}{2} & \frac{79}{6} & \frac{85}{6} & \frac{121}{6} & \frac{209}{6} \\[0.1cm]
\frac{13}{6} & \frac{31}{6} & \frac{37}{6} & \frac{79}{6} & \frac{85}{6} & \frac{121}{6} & \frac{211}{6} & \frac{205}{6} & \frac{265}{6} & \frac{433}{6} \\[0.1cm]
\frac{17}{6} & \frac{37}{6} & \frac{19}{2} & \frac{85}{6} & \frac{121}{6} & \frac{209}{6} & \frac{205}{6} & \frac{265}{6} & \frac{433}{6} & \frac{267}{2} \\[0.1cm]
\frac{31}{6} & \frac{79}{6} & \frac{85}{6} & \frac{211}{6} & \frac{205}{6} & \frac{265}{6} & \frac{583}{6} & \frac{517}{6} & \frac{601}{6} & \frac{913}{6} \\[0.1cm]
\frac{37}{6} & \frac{85}{6} & \frac{121}{6} & \frac{205}{6} & \frac{265}{6} & \frac{433}{6} & \frac{517}{6} & \frac{601}{6} & \frac{913}{6} & \frac{1633}{6} \\[0.1cm]
\frac{19}{2} & \frac{121}{6} & \frac{209}{6} & \frac{265}{6} & \frac{433}{6} & \frac{267}{2} & \frac{601}{6} & \frac{913}{6} & \frac{1633}{6} & \frac{3137}{6} \\[0.1cm]
\frac{79}{6} & \frac{211}{6} & \frac{205}{6} & \frac{583}{6} & \frac{517}{6} & \frac{601}{6} & \frac{1651}{6} & \frac{1357}{6} & \frac{1417}{6} & \frac{1969}{6} \\[0.1cm]
\frac{85}{6} & \frac{205}{6} & \frac{265}{6} & \frac{517}{6} & \frac{601}{6} & \frac{913}{6} & \frac{1357}{6} & \frac{1417}{6} & \frac{1969}{6} & \frac{3361}{6} \\[0.1cm]
\frac{121}{6} & \frac{265}{6} & \frac{433}{6} & \frac{601}{6} & \frac{913}{6} & \frac{1633}{6} & \frac{1417}{6} & \frac{1969}{6} & \frac{3361}{6} & \frac{6337}{6} \\[0.1cm]
\frac{209}{6} & \frac{433}{6} & \frac{267}{2} & \frac{913}{6} & \frac{1633}{6} & \frac{3137}{6} & \frac{1969}{6} & \frac{3361}{6} & \frac{6337}{6} & \frac{4139}{2}

\end{bmatrix}
\cdot
\begin{bmatrix}
    0\\[0.1cm]
    14\\[0.1cm]
    0\\[0.1cm]
    -21\\[0.1cm]
    2\\[0.1cm]
    0\\[0.1cm]
    5\\[0.1cm]
    0\\[0.1cm]
    0\\[0.1cm]
    0
\end{bmatrix}
= 
\begin{bmatrix}
    0\\[0.1cm]
    0\\[0.1cm]
    0\\[0.1cm]
    0\\[0.1cm]
    0\\[0.1cm]
    0\\[0.1cm]
    0\\[0.1cm]
    0\\[0.1cm]
    0\\[0.1cm]
    0
\end{bmatrix}.
$$

Alternatively, we could have generated $M(\mu)(3)$ from $M(\mu)(2)$ after finding a recursive relation where each degree 3 monomial is the leading monomial with respect to lexicographic ordering. For example, using the ideal-like property we have the recursive relations: 
\begin{align*}
    &14x - 21 x^2 +2xy+ 5x^3,\\
    &12x - 8 x^2 -9xy+ 5x^2y, \\
    &12y - 8 xy -9y^2+ 5xy^2, \text{ and} \\
    &16y +6 xy -27y^2+ 5y^3.
\end{align*}

Thus let 

$$C:=\begin{bmatrix}
    0 & \frac{-14}{5} & 0 &\frac{21}{5} & \frac{-2}{5} & 0\\[0.1cm]
    0 & \frac{-12}{5} & 0 &\frac{8}{5} & \frac{9}{5} & 0\\[0.1cm]
    0 & 0 & \frac{-12}{5} &0 & \frac{8}{5} & \frac{9}{5}\\[0.1cm]
    0 & 0 & \frac{-16}{5} &0 & \frac{-6}{5} & \frac{27}{5}
\end{bmatrix},$$
so then

$$M(\mu)(3) = \begin{bmatrix}
    M(\mu)(2) & M(\mu)(2) \cdot C^T\\
    C\cdot M(\mu)(2)& C\cdot M(\mu)(2) \cdot C^T
\end{bmatrix}.$$

Note the changing of the signs for the entries in $C.$ This due to the recursive relations themselves will equal zero,  $14X - 21 X^2 +2XY+ 5X^3 = \mathbf{0}$, thus to generate the columns $\frac{-14}{5}X + \frac{21}{5} X^2 -\frac{2}{5}XY=X^3.$ 

Next, let $\nu:= \frac{1}{3}\delta(2,4)+\frac{1}{3}\delta(3,2) + \frac{1}{3}\delta(1,1)$. Then by Corollary 3.5, $M(\nu)(3)$ and $M(\mu)(3)$ will have the same recursive relations. Observe $M(\nu)(3) \cdot \hat{p} = \mathbf{0}$ where $\hat{p}$ is the same as above.
$$
\begin{bmatrix}
1 & 2 & \frac{7}{3} & \frac{14}{3} & 5 & 7 & 12 & \frac{35}{3} & 15 & \frac{73}{3} \\[0.1cm]
2 & \frac{14}{3} & 5 & 12 & \frac{35}{3} & 15 & \frac{98}{3} & 29 & \frac{101}{3} & 51 \\[0.1cm]
\frac{7}{3} & 5 & 7 & \frac{35}{3} & 15 & \frac{73}{3} & 29 & \frac{101}{3} & 51 & 91 \\[0.1cm]
\frac{14}{3} & 12 & \frac{35}{3} & \frac{98}{3} & 29 & \frac{101}{3} & 92 & \frac{227}{3} & 79 & \frac{329}{3} \\[0.1cm]
5 & \frac{35}{3} & 15 & 29 & \frac{101}{3} & 51 & \frac{227}{3} & 79 & \frac{329}{3} & 187 \\[0.1cm]
7 & 15 & \frac{73}{3} & \frac{101}{3} & 51 & 91 & 79 & \frac{329}{3} & 187 & \frac{1057}{3} \\[0.1cm]
12 & \frac{98}{3} & 29 & 92 & \frac{227}{3} & 79 & \frac{794}{3} & 205 & \frac{581}{3} & 243 \\[0.1cm]
\frac{35}{3} & 29 & \frac{101}{3} & \frac{227}{3} & 79 & \frac{329}{3} & 205 & \frac{581}{3} & 243 & \frac{1169}{3} \\[0.1cm]
15 & \frac{101}{3} & 51 & 79 & \frac{329}{3} & 187 & \frac{581}{3} & 243 & \frac{1169}{3} & 715 \\[0.1cm]
\frac{73}{3} & 51 & 91 & \frac{329}{3} & 187 & \frac{1057}{3} & 243 & \frac{1169}{3} & 715 & 1387

\end{bmatrix}
\cdot
\begin{bmatrix}
    0\\[0.1cm]
    14\\[0.1cm]
    0\\[0.1cm]
    -21\\[0.1cm]
    2\\[0.1cm]
    0\\[0.1cm]
    5\\[0.1cm]
    0\\[0.1cm]
    0\\[0.1cm]
    0
\end{bmatrix}
= 
\begin{bmatrix}
    0\\[0.1cm]
    0\\[0.1cm]
    0\\[0.1cm]
    0\\[0.1cm]
    0\\[0.1cm]
    0\\[0.1cm]
    0\\[0.1cm]
    0\\[0.1cm]
    0\\[0.1cm]
    0
\end{bmatrix}.
$$

Now using Proposition 2.7 we can construct a Gröbner basis for the recursive relations for both $M(\mu)$ and $M(\nu)$; explicitly, the Gröbner basis is $\{8 - 14 y + 7 y^2 - y^3, 16 + 6 x - 27 y + 5 y^2\}.$

\bibliography{mybibliography.bib}
\bibliographystyle{amsplain}

\end{document}